\documentclass[12pt, reqno]{amsart}

\usepackage{url}
\usepackage{amsfonts}
\usepackage{amsthm}
\usepackage[psamsfonts]{amssymb}
\usepackage[dvipdfmx]{graphicx}
\usepackage{ascmac}
\usepackage{amsmath}
\usepackage{fancybox}
\usepackage{fancyhdr}
\usepackage{enumerate}
\usepackage{mathrsfs}
\usepackage{color}
\usepackage{multicol}
\usepackage{framed}

\theoremstyle{plain}
\newtheorem{thm}[equation]{Theorem}

\newtheorem{cor}[equation]{Corollary}
\newtheorem{rem}[equation]{\it Remark}

\newtheorem{lem}[equation]{Lemma}

\newtheorem*{Acknowledgements}{\it Acknowledgements}

\keywords{Bakry-\'{E}mery Ricci curvature, Index form, Upper diameter bound, Ricci soliton, Hitchin-Thorpe inequality}
\subjclass[2010]{Primary 53C21, Secondary 53C20, 53C25}
\address{Department of Mathematics, Graduate School of Science, Osaka University, 1-1 Machikaneyama, Toyonaka, Osaka 560-0043, JAPAN}
\email{h-tadano@cr.math.sci.osaka-u.ac.jp}

\address{Faculty of Mathematics, University of Santiago de Compostela, 15782 Santiago de Compostela, SPAIN}
\email{homare.tadano@usc.es}

\title{Remark on a diameter bound for complete manifolds with positive Bakry-\'{E}mery Ricci curvature}
\dedicatory{Dedicated to Professor Kimio Miyajima on the occasion of his retirement}
\author{Homare TADANO}
\date{submitted July 22, 2015, revised October 27, 2015}
\thanks{This work was supported by Moriyasu Graduate Student Scholarship Foundation}

\begin{document}

\begin{abstract}
In this paper, we shall give a new upper diameter estimate for complete Riemannian manifolds in the case that the Bakry-\'{E}mery Ricci curvature has a positive lower bound and the norm of the potential function has an upper bound. Our diameter estimate improves previous ones obtained by Wei and Wylie (J. Differential Geom. 83, 377--405, 2009) and Limoncu (Math. Z. 271, 715--722, 2012). As an application, we shall give an upper diameter bound for compact Ricci solitons in terms of the maximum value of the scalar curvature. By using such a diameter bound, we shall provide some new sufficient conditions for four-dimensional compact Ricci solitons to satisfy the Hitchin-Thorpe inequality.
\end{abstract}

\maketitle

\numberwithin{equation}{section}

\section{Introduction}

Let $(M, g)$ be a complete Riemannian manifold and $f : M \rightarrow \mathbb{R}$ a smooth function. A \textit{Bakry-\'{E}mery Ricci curvature} \cite{B-E} is defined by $\operatorname{Ric}_{g} + \operatorname{Hess} f$, where $\operatorname{Ric}_{g}$ stands the Ricci curvature of $(M, g)$ and $\operatorname{Hess} f$ denotes the Hessian of $f$. Recently, the Bakry-\'{E}mery Ricci curvature has received much attention in various areas of mathematics, since it is a good substitute for the Ricci curvature allowing us to establish many interesting theorems in metric measure spaces, such as comparison theorems \cite{W-W}, eigenvalue estimates \cite{F-L-L}, Li-Yau Harnack inequalities \cite{Li}. In particular, Wei and Wylie \cite{W-W} proved the following Myers type theorem via Bakry-\'{E}mery Ricci curvature:

\begin{thm}[Wei-Wylie \cite{W-W}]
Let $(M, g)$ be an $n$-dimensional complete connected Riemannian manifold satisfying
\[
\operatorname{Ric}_{g} + \operatorname{Hess} f \geqslant (n - 1)Hg
\]
for some constant $H > 0$. If $| f | \leqslant k$ for some constant $k \geqslant 0$, then $M$ must be compact. Moreover, 
\begin{equation}\label{diam-1}
\operatorname{diam}(M, g) \leqslant \frac{\pi}{\sqrt{H}} + \frac{4k}{(n - 1) \sqrt{H}}.
\end{equation}
\end{thm}

On the other hand, Limoncu \cite{L} gave the following diameter estimate for complete Riemannian manifolds under the same assumption as in the previous theorem:

\begin{thm}[Limoncu \cite{L}]\label{Thm-Limoncu}
Let $(M, g)$ be an $n$-dimensional complete connected Riemannian manifold satisfying
\[
\operatorname{Ric}_{g} + \operatorname{Hess} f \geqslant (n - 1)Hg
\]
for some constant $H > 0$. If $| f | \leqslant k$ for some constant $k \geqslant 0$, then $M$ must be compact. Moreover, 
\begin{equation}\label{diam-2}
\operatorname{diam}(M, g) \leqslant \frac{\pi}{\sqrt{H}} \sqrt{1 + \frac{2 \sqrt{2} k}{n - 1}}.
\end{equation}
In particular, if $k \geqslant \frac{(n - 1) \pi}{8}( \sqrt{2} \pi - 4)$, then the estimate {\rm (\ref{diam-2})} is sharper than {\rm (\ref{diam-1})}.
\end{thm}

The aim of this paper is to improve these two diameter estimates under the same assumptions as in the two previous theorems by giving the following:

\begin{thm}\label{Main-Theorem}
Let $(M, g)$ be an $n$-dimensional complete connected Riemannian manifold satisfying
\begin{equation}\label{assumption}
\operatorname{Ric}_{g} + \operatorname{Hess} f \geqslant (n - 1)Hg
\end{equation}
for some constant $H > 0$. If $| f | \leqslant k$ for some constant $k \geqslant 0$, then $M$ must be compact. Moreover, 
\begin{equation}\label{diam-result}
\operatorname{diam}(M, g) \leqslant \frac{\pi}{\sqrt{H}} \sqrt{1 + \frac{8k}{(n - 1) \pi}}.
\end{equation}
\end{thm}

\begin{rem}\rm
Since
\[
\frac{8}{\pi} \approx 2.54647 \quad \mbox{and} \quad 2 \sqrt{2} \approx 2.82842, 
\]
our diameter estimate (\ref{diam-result}) is sharper than (\ref{diam-2}). Moreover, we may easily see that our estimate (\ref{diam-result}) is also sharper than (\ref{diam-1}) \textit{without any assumptions on $k$}.
\end{rem}

Our Theorem \ref{Main-Theorem} has applications to an upper diameter bound and the Hitchin-Thorpe inequality for compact Ricci solitons. A complete Riemannian manifold $(M, g)$ is called a \textit{Ricci soliton} if there exists a vector field $X \in \mathfrak{X}(M)$ satisfying the equation
\begin{equation}\label{RS}
\operatorname{Ric}_{g} + \frac{1}{2} \mathcal{L}_{X} g = \lambda g
\end{equation}
for some constant $\lambda \in \mathbb{R}$, where $\mathcal{L}_{X}$ denotes the Lie derivative by $X$. We say that the soliton $(M, g)$ is \textit{shrinking}, \textit{steady} and \textit{expanding} described as $\lambda > 0, \lambda = 0$ and $\lambda < 0$, respectively. Note that if $X$ is a Killing vector field, then the soliton is an Einstein manifold. In such a case, we say that the soliton is \textit{trivial}. When $X$ may be replaced with a gradient vector field $\nabla f$ for some smooth function $f : M \rightarrow \mathbb{R}$, called a \textit{potential function}, we call $(M, g)$ a \textit{gradient Ricci soliton}. Then (\ref{RS}) becomes
\begin{equation}\label{GRS}
\operatorname{Ric}_{g} + \operatorname{Hess} f = \lambda g.
\end{equation}
Thanks to Perelman \cite{P}, any compact Ricci soliton is gradient. It is known \cite{Cao} that any non-trivial compact Ricci soliton $(M, g)$ is shrinking with $\dim M \geqslant 4$. Moreover, it is also known \cite{Cao} that the potential function $f$ of any gradient Ricci soliton $(M, g)$ satisfies $R + | \nabla f |^{2} - 2 \lambda f = C$ for some real constant $C$, where $R$ denotes the scalar curvature on the soliton. By adding some constant on $f$, we may normalize $f$ such that
\begin{equation}\label{normalize}
R + | \nabla f |^{2} = 2 \lambda f.
\end{equation}
Fern\'{a}ndez-L\'{o}pez and Garc\'{i}a-R\'{i}o \cite{FL-GR} investigated a lower diameter bound for compact shrinking Ricci solitons depending on the scalar and Ricci curvatures. By using Theorem \ref{Main-Theorem}, we shall then give an upper diameter bound for compact shrinking Ricci solitons in terms of the maximum value of the scalar curvature.

\begin{cor}\label{Cor-1}
Let $(M, g)$ be an $n$-dimensional compact connected shrinking Ricci soliton satisfying {\rm (\ref{GRS})}. Suppose that the soliton is normalized in sense of {\rm (\ref{normalize})}. Then
\begin{equation}\label{diam-Cor-1}
\operatorname{diam}(M, g) \leqslant \frac{\pi}{\sqrt{\lambda}} \sqrt{n - 1 + \frac{4 R_{\mathrm{max}}}{\pi \lambda}}, 
\end{equation}
where $R_{\mathrm{max}}$ denotes the maximum value of the scalar curvature $R$ on the soliton.
\end{cor}

Since Ricci solitons are natural generalization of Einstein manifolds, we may expect some topological obstructions to the existence of compact Ricci solitons. The Hitchin-Thorpe inequality for compact shrinking Ricci solitons was proved by Ma \cite{Ma} assuming some upper bounds on the $L^{2}$-norm of the scalar curvature, while Fern\'{a}ndez-L\'{o}pez and Garc\'{i}a-R\'{i}o \cite{FL-GR} investigated the Hitchin-Thorpe inequality assuming some upper diameter bounds in terms of the Ricci curvature. By using Corollary \ref{Cor-1}, we then provide the following new sufficient condition for four-dimensional compact shrinking Ricci solitons to satisfy the Hitchin-Thorpe inequality:

\begin{cor}\label{Cor-2}
Let $(M, g)$ be a four-dimensional compact connected shrinking Ricci soliton satisfying {\rm (\ref{GRS})}. Suppose that the soliton is normalized in sense of {\rm (\ref{normalize})}. If
\begin{equation}\label{diam-Cor-2}
\sqrt{\frac{R_{\mathrm{max}}}{\lambda^{2}} \left( 4 \pi + \frac{\pi^{2}}{2} \right)} \leqslant \operatorname{diam}(M, g), 
\end{equation}
then the soliton satisfies the Hitchin-Thorpe inequality $2 \chi(M) \geqslant 3 | \tau(M) |$.
\end{cor}

This paper is organized as follows: In Section 2, after introducing our notation, we shall give a proof of Theorem \ref{Main-Theorem}. Ending with Section 3, proofs of Corollary \ref{Cor-1} and \ref{Cor-2} shall be given.

\begin{Acknowledgements}\rm
The author would like to thank Professor Toshiki Mabuchi for his encouragements. This work was carried out while the author was visiting University of Santiago de Compostela. The author would also like to thank Professor Eduardo Garc\'{i}a-R\'{i}o for his warm hospitality.
\end{Acknowledgements}

\section{A proof of Theorem \ref{Main-Theorem}}

Before giving a proof of Theorem \ref{Main-Theorem}, we shall introduce our notation. Let $X, Y, Z \in \mathfrak{X}(M)$ be three vector fields on $M$. For any smooth function $f \in \mathcal{C}^{\infty}(M)$, the gradient vector field and Hessian of $f$ are defined by
\[
g( \nabla f, X) = df(X) \quad \mbox{and} \quad \operatorname{Hess}f (X, Y) = g(\nabla_{X} \nabla f, Y), 
\]
respectively. The curvature tensor and Ricci tensor are defined by
\[
R(X, Y)Z = \nabla_{X} \nabla_{Y} Z - \nabla_{Y} \nabla_{X} Z - \nabla_{[X, Y]} Z \quad \mbox{and} \quad \operatorname{Ric}_{g}(X, Y) = \sum_{i = 1}^{n} g(R(e_{i}, X)Y, e_{i}), 
\]
respectively. Here, $\{ e_{i} \}_{i = 1}^{n}$ is an orthonormal frame of $(M, g)$. In order to prove Theorem \ref{Main-Theorem}, we shall use the index form of a minimizing unit speed geodesic segment. We refer the reader to the books \cite{Lee, Petersen} for basic facts about this topic.

\begin{proof}[Proof of Theorem {\rm \ref{Main-Theorem}}]
Our proof of Theorem \ref{Main-Theorem} is similar to that of Theorem \ref{Thm-Limoncu} by Limoncu \cite{L}. Take arbitrary two points $p, q \in M$. By the compactness of the manifold $(M, g)$, there exists a minimizing unit speed geodesic segment $\sigma$ from $p$ to $q$ of length $\ell$. Let $\{ e_{1} = \dot{\sigma}, e_{2}, \cdots, e_{n} \}$ be a parallel orthonormal frame along $\sigma$. Recall that, for any smooth function $\phi \in \mathcal{C}^{\infty}([0, \ell])$ satisfying $\phi(0) = \phi(\ell) = 0$, we obtain
\begin{equation}\label{index}
I(\phi e_{i}, \phi e_{i}) = \int_{0}^{\ell} \left( g(\dot{\phi} e_{i}, \dot{\phi} e_{i}) - g(R(\phi e_{i}, \dot{\sigma}) \dot{\sigma}, \phi e_{i}) \right) dt, 
\end{equation}
where $I(\cdot, \cdot)$ denotes the index form of $\sigma$. From (\ref{index}), we have
\begin{equation}\label{index-sum}
\sum_{i = 2}^{n} I(\phi e_{i}, \phi e_{i}) = \int_{0}^{\ell} \left( (n - 1) \dot{\phi}^{2} - \phi^{2} \operatorname{Ric}_{g}(\dot{\sigma}, \dot{\sigma}) \right) dt, 
\end{equation}
where we have used $g(R(\dot{\sigma}, \dot{\sigma})\dot{\sigma}, \dot{\sigma}) = 0$. By using the assumption (\ref{assumption}) in the integral expression (\ref{index-sum}), we obtain
\begin{align}
\sum_{i = 2}^{n} I(\phi e_{i}, \phi e_{i}) & \leqslant \int_{0}^{\ell} \left( (n - 1)(\dot{\phi}^{2} - H \phi^{2}) + \phi^{2} \operatorname{Hess} f(\dot{\sigma}, \dot{\sigma}) \right) dt \nonumber \\
& = \int_{0}^{\ell} \left( (n - 1)(\dot{\phi}^{2} - H \phi^{2}) + \phi^{2} g(\nabla_{\dot{\sigma}} \nabla f, \dot{\sigma}) \right) dt \nonumber \\
& = \int_{0}^{\ell} \left( (n - 1)(\dot{\phi}^{2} - H \phi^{2}) + \phi^{2} \dot{\sigma}(g(\nabla f, \dot{\sigma})) \right) dt, \label{eq1}
\end{align}
where the last equality follows from the parallelism of the metric $g$ and $\nabla_{\dot{\sigma}} \dot{\sigma} = 0$. On the geodesic segment $\sigma(t)$, we have
\begin{align}
\phi^{2} \dot{\sigma} \left( g(\nabla f, \dot{\sigma}) \right) & = \phi^{2} \frac{d}{dt}(g(\nabla f, \dot{\sigma})) \nonumber \\
& = - 2 \phi \dot{\phi} g(\nabla f, \dot{\sigma}) + \frac{d}{dt}(\phi^{2} g(\nabla f, \dot{\sigma})) \nonumber \\
& = 2 f \frac{d}{dt} (\phi \dot{\phi}) - 2 \frac{d}{dt} (f \phi \dot{\phi}) + \frac{d}{dt}(\phi^{2} g(\nabla f, \dot{\sigma})), \label{eq2}
\end{align}
where in the last equality, we have used $g(\nabla f, \dot{\sigma}) = \frac{df}{dt}(\sigma(t))$. Hence, by integrating both sides of (\ref{eq2}), we have
\begin{align}
\int_{0}^{\ell} \phi^{2} \dot{\sigma}( g(\nabla f, \dot{\sigma})) dt & = \int_{0}^{\ell} 2f \frac{d}{dt}(\phi \dot{\phi}) dt - 2 \left[ f \phi \dot{\phi} \right]_{0}^{\ell} + \left[ \phi^{2} g(\nabla f, \dot{\sigma}) \right]_{0}^{\ell} \nonumber \\
& = 2 \int_{0}^{\ell} f \frac{d}{dt}(\phi \dot{\phi}) dt, \label{eq3}
\end{align}
where the last equality follows from $\phi(0) = \phi(\ell) = 0$. By (\ref{eq3}) and the assumption $| f | \leqslant k$ in Theorem \ref{Main-Theorem}, 
we obtain
\begin{equation}\label{eq4}
\int_{0}^{\ell} \phi^{2} \dot{\sigma}( g(\nabla f, \dot{\sigma})) dt \leqslant 2k \int_{0}^{\ell} \left| \frac{d}{dt}(\phi \dot{\phi}) \right| dt.
\end{equation}
From (\ref{eq1}) and (\ref{eq4}), we have
\begin{equation}\label{eq-main}
\sum_{i = 2}^{n} I(\phi e_{i}, \phi e_{i}) \leqslant \int_{0}^{\ell} (n - 1)(\dot{\phi}^{2} - H \phi^{2}) dt + 2k \int_{0}^{\ell} \left| \frac{d}{dt} (\phi \dot{\phi}) \right| dt.
\end{equation}
If the funtion $\phi$ is taken to be $\phi(t) = \sin (\frac{\pi t}{\ell})$, then we obtain $\dot{\phi}(t) = \frac{\pi}{\ell} \cos (\frac{\pi t}{\ell})$ and
\[
\phi \dot{\phi} = \frac{\pi}{\ell} \sin \left( \frac{\pi t}{\ell} \right) \cos \left( \frac{\pi t}{\ell} \right) = \frac{\pi}{2 \ell} \sin \left( \frac{2 \pi t}{\ell} \right).
\]
Then, (\ref{eq-main}) becomes
\[
\begin{aligned}
\sum_{i = 2}^{n} I(\phi e_{i}, \phi e_{i}) & \leqslant (n - 1) \int_{0}^{\ell} \left( \frac{\pi^{2}}{\ell^{2}} \cos^{2} \left( \frac{\pi t}{\ell} \right) - H \sin^{2} \left( \frac{\pi t}{\ell} \right) \right) dt \\
& \quad + 2k \left( \frac{\pi}{\ell} \right)^{2} \int_{0}^{\ell} \left| \cos \frac{2 \pi t}{\ell} \right| dt, 
\end{aligned}
\]
and consequently, we have
\[
\sum_{i = 2}^{n} I(\phi e_{i}, \phi e_{i}) \leqslant - \frac{1}{2 \ell} \left( (n - 1) H \ell^{2} - (n - 1) \pi^{2} - 8k \pi \right).
\]
Since $\sigma$ is a minimizing geodesic, we must obtain
\[
(n - 1) H \ell^{2} - (n - 1) \pi^{2} - 8k \pi \leqslant 0.
\]
From this inequality, we have
\[
\ell \leqslant \frac{\pi}{\sqrt{H}} \sqrt{1 + \frac{8k}{(n - 1) \pi}}.
\]
This proves Theorem \ref{Main-Theorem}.
\end{proof}

\begin{rem}\rm
By using the Cauchy-Schwarz inequality, Limoncu \cite{L} estimated (\ref{eq3}) from above by
\[
\int_{0}^{\ell} \phi^{2} \dot{\sigma}( g(\nabla f, \dot{\sigma})) dt = 2 \int_{0}^{\ell} f \frac{d}{dt}(\phi \dot{\phi}) dt \leqslant 2 \sqrt{\int_{0}^{\ell} f^{2} dt} \sqrt{\int_{0}^{\ell} \left( \frac{d}{dt} (\phi \dot{\phi}) \right)^{2} dt}, 
\]
while we estimated (\ref{eq3}) from above by its absolute value in (\ref{eq4}) and obtained the better estimate (\ref{diam-result}) than (\ref{diam-2}).
\end{rem}

\section{Applications to Theorem {\rm \ref{Main-Theorem}}}

In this section, by using Theorem \ref{Main-Theorem}, we shall give proofs of Corollary \ref{Cor-1} and \ref{Cor-2}. Throughout this section, we assume that $(M, g)$ is a four-dimensional compact connected normalized shrinking Ricci soliton satisfying (\ref{GRS}) and (\ref{normalize}).

\subsection{A proof of Corollary \ref{Cor-1}}

The following lemma is useful to prove Corollary \ref{Cor-1}:

\begin{lem}\label{Lem}
The potential function $f$ on the soliton $(M, g)$ satisfies
\[
0 \leqslant 2 \lambda f \leqslant R_{\mathrm{max}}, 
\]
where $R_{\mathrm{max}}$ denotes the maximum value of the scalar curvature $R$ on the soliton.
\end{lem}

\begin{proof}
Thanks to Chen \cite{Chen}, the scalar curvature of any complete shrinking Ricci soliton is non-negative. Hence, by (\ref{normalize}), we have $2 \lambda f \geqslant 0$. On the other hand, by compactness of the manifold $M$, there exists some global maximum point $p \in M$ of the potential function. Then, it follows from (\ref{normalize}) that for any point $x \in M$, 
\[
2 \lambda f(p) = R(p) \geqslant 2 \lambda f(x) = R(x) + | \nabla f |^{2} (x), 
\]
and hence, $R(p) \geqslant R(x)$. Therefore, the scalar curvature also attains its maximum at $p$, and we obtain the result.
\end{proof}

Corollary \ref{Cor-1} follows immediately from Theorem \ref{Main-Theorem} and Lemma \ref{Lem}.

\subsection{A proof of Corollary \ref{Cor-2}}

We use the following theorem to prove Corollary \ref{Cor-2}:

\begin{thm}[Ma \cite{Ma}]\label{Thm-Ma}
Let $(M, g)$ be a four-dimensional compact shrinking Ricci soliton satisfying {\rm (\ref{GRS})}. If the scalar curvature $R$ satisfies
\[
\int_{M} R^{2} \leqslant 24 \lambda^{2} \mathrm{vol}(M, g), 
\]
then the soliton $(M, g)$ satisfies the Hitchin-Thorpe inequality $2 \chi(M) \geqslant 3 | \tau(M) |$.
\end{thm}

\begin{proof}[Proof of Corollary {\rm \ref{Cor-2}}]
By taking the trace of (\ref{GRS}), we have
\begin{equation}\label{Cor-2-eq1}
R + \Delta f = 4 \lambda.
\end{equation}
Thanks to Theorem \ref{Main-Theorem}, the diameter of $(M, g)$ has the upper bound
\begin{equation}\label{diam-Cor-2-eq2}
\operatorname{diam}(M, g) \leqslant \frac{\pi}{\sqrt{\lambda}} \sqrt{3 + \frac{4 R_{\mathrm{max}}}{\pi \lambda}}.
\end{equation}
Suppose that the inequality (\ref{diam-Cor-2}) holds. Then, by (\ref{diam-Cor-2-eq2}) we obtain
\[
\frac{R_{\mathrm{max}}}{\lambda^{2}} \left( 4 \pi + \frac{\pi^{2}}{2} \right) \leqslant \operatorname{diam}^{2}(M, g) \leqslant \frac{\pi^{2}}{\lambda} \left( 3 + \frac{4 R_{\mathrm{max}}}{\pi \lambda} \right), 
\]
from where we have $R_{\mathrm{max}} \leqslant 6 \lambda$. Since the scalar curvature of any complete shrinking Ricci soliton is non-negative, it follows from (\ref{Cor-2-eq1}) that
\[
\int_{M} R^{2} \leqslant R_{\mathrm{max}} \int_{M} R = 24 \lambda^{2} \mathrm{vol}(M, g), 
\]
and the result follows from Theorem \ref{Thm-Ma}.
\end{proof}

By using the same way as in the previous proof, we may easily show the following:

\begin{cor}\label{Cor-3}
Let $(M, g)$ be a four-dimensional compact connected shrinking Ricci soliton satisfying {\rm (\ref{GRS})}. Suppose that the soliton is normalized in sense of {\rm (\ref{normalize})}. If
\begin{equation}\label{diam-Cor-3}
\frac{R_{\mathrm{max}}}{6 \lambda} \cdot \frac{\pi}{\sqrt{\lambda}} \sqrt{3 + \frac{4 R_{\mathrm{max}}}{\pi \lambda}} \leqslant \operatorname{diam}(M, g), 
\end{equation}
then the soliton satisfies the Hitchin-Thorpe inequality $2 \chi(M) \geqslant 3 | \tau(M) |$.
\end{cor}

\begin{rem}\rm
The inequality {\rm (\ref{diam-Cor-3})} may be a better condition than (\ref{diam-Cor-2}) when the maximum value $R_{\mathrm{max}}$ of the scalar curvature is sufficiently small.
\end{rem}

\end{document}